\documentclass{amsart}    
\usepackage{graphics}
\usepackage{amssymb}

\newcommand{\Q}{\mathbb{Q}}
\newcommand{\C}{\mathbb{C}}
\newcommand{\D}{\mathbb{D}}

\newcommand{\R}{\mathbb{R}}
\newcommand{\N}{\mathbb{N}}
\newcommand{\XTC}{\hat{\mathbb{C}}}

\newcommand{\dom}{\mbox{dom}}
\newcommand{\ran}{\mbox{ran}}

\newcommand{\normal}{n}

\newcommand{\Int}{\mbox{Int}}
\newcommand{\Ext}{\mbox{Ext}}

\newtheorem{theorem}{Theorem}[section]
\newtheorem{lemma}[theorem]{Lemma}

\theoremstyle{definition}
\newtheorem{definition}[theorem]{Definition}

\theoremstyle{theorem}
\newtheorem{corollary}[theorem]{Corollary}

\theoremstyle{theorem}

\theoremstyle{theorem}

\theoremstyle{theorem}

\theoremstyle{definition}

\theoremstyle{theorem}

\numberwithin{equation}{section}

\begin{document}

\title{Computing conformal maps onto circular domains}

\author{Valentin V. Andreev}
\address{Department of Mathematics\\
              Lamar University\\
              Beaumont, Texas 77710 USA}
 \email{valentin.andreev@lamar.edu}                          

\author{Dale Daniel}
\address{Department of Mathematics\\
              Lamar University\\
              Beaumont, Texas 77710 USA}
\email{dale.daniel@lamar.edu}

\author{Timothy H. McNicholl}
\address{Department of Mathematics\\
              Lamar University\\
              Beaumont, Texas 77710 USA}
\email{timothy.h.mcnicholl@gmail.com}

\begin{abstract}
We show that, given a non-degenerate, finitely connected domain $D$, its boundary, and the number of its boundary components, it is possible to compute a conformal mapping of $D$ onto a circular domain \emph{without} prior knowledge of the circular
domain.  We do so by computing a suitable bound on the error in the Koebe construction
(but, again, without knowing the circular domain in advance).  As a scientifically sound model of computation with continuous data, we use
Type-Two Effectivity\cite{Weihrauch.2000}.
\end{abstract}
\keywords{Computable analysis, constructive analysis, complex analysis, conformal mapping}
\subjclass[2010]{03F60, 30C20, 30C30, 30C85}

\maketitle

\section{Introduction}\label{sec:INTRO}

Let $\C$ denote the complex plane, and let $\XTC = \XTC \cup \{\infty\}$ denote the extended complex plane.  A \emph{domain} is an open and connected subset of $\XTC$.
A domain is \emph{$n$-connected} if its complement has $n$ connected components.  The boundaries of these components are referred to as the \emph{boundary components} of the domain.  
A domain is \emph{finitely connected} if it is $n$-connected for some $n$.  
A domain is \emph{circular} if its boundary components are circles.
Finally, a domain is \emph{non-degenerate} if each connected component of its complement   
contains at least two points.

The Riemann Mapping Theorem states
that every non-degenerate $1$-connected domain is conformally equivalent to 
the unit disk, $\D$.  Hence, there is a single \emph{canonical domain} to which all
non-degenerate $1$-connected domains are conformally equivalent.  With regards to the situation for $2$-connected domains and 
higher, it is not the case that all $n$-connected domains
are conformally equivalent when $n \geq 2$.  In fact, two annuli are conformally
equivalent only when the ratio of their inner to outer radii are the same.  Hence, much research has focused on proving 
existence of conformal mappings onto various kinds of canonical domains.  
See \cite{Nehari.1952} and \cite{Golusin.1969}.  Some existence proofs rely on extremal-value arguments and hence are not constructive.  At the same time, there are explicit formulae for the conformal
mappings from a multiply connected domain to a circularly slit disk
and a circularly slit annulus. Unlike the simply connected case, every
multiply connected domain can be identified by $3n-6$ parameters,
where $n$ is the connectivity of the domain. But even if we know one
type canonical domain to which a given domain is conformally
equivalent, \it we are still not wiser about the configuration of other
canonical domains to which the given domain is conformally
equivalent.\rm

Among the canonical domains are the \emph{circular domains} which are obtained by deleting one or more disjoint closed disks from the plane.  
These domains are the canonical domains in a number of
recent studies of the Schwarz-Christoffel formula for multiply
connected domains, nonlinear problems in mechanics, and in aircraft
engineering.  See, for example, \cite{Crowdy.2007} and \cite{DeLillo.Elcrat.Pfaltzgraff.2004}.
Jeong and Mityushev have recently obtained explicit formulae for the Green's function of a circular domain \cite{Jeong.Mityushev.2007}.

One key motivation for this paper is the following theorem was proven by P. Koebe in 1910 \cite{Koebe.1910}.

\begin{theorem}\label{thm:KOEBE}
Suppose $D$ is a non-degenerate and finitely connected domain that contains $\infty$ but not $0$.  Then, there is a unique circular domain $C_D$ for which there exists a unique conformal map $f_D$ of $D$ onto $C_D$ such that 
$f_D(z) = z + O(z^{-1})$.  
\end{theorem}

In the same paper, Paul Koebe put forth and investigated an iterative technique for approximating the unique conformal map $f_D$ of a non-degenerate $n$-connected domain $D$ with $\infty \in D$ and $0 \not \in D$ onto a circular domain $C_D$ of the form $z + O(z^{-1})$ \cite{Koebe.1910}.   This method
is now known as the \emph{Koebe Construction} and will be described in Section \ref{sec:KOEBE}.  
In 1959, Dieter Gaier calculated upper bounds on the error in this construction
which tend to zero as the iterations progress \cite{Gaier.1959}.  However, these
bounds use certain numbers associated with $C_D$.   Hence, to apply Gaier's bounds
for the sake of approximating $f_D$, one must first know a fair amount about the circular domain $C_D$.  In \cite{Henrici.1986}, Henrici presents a modification of Gaier's construction.  But, again, Henrici's bounds use certain numbers associated with the circular domain $C_D$ which usually is not known in advance.  

The purpose of this paper is to show that recent results in complex analysis by the first and third author lead naturally to a proof of the effective version of Theorem \ref{thm:KOEBE} within the framework of Type-2 Effectivity (TTE) \cite{McNicholl.2011.b}, \cite{Andreev.McNicholl.2011}, \cite{Weihrauch.2000}.   
That is, informally speaking, we show that arbitrarily good approximations to $f_D(z)$ and $C_D$ can be computed from sufficiently good approximations to $z$, $D$, and the boundary of $D$.  Furthermore, we show that this result is optimal in that arbitrarily good approximations to $D$ and its boundary can be computed from sufficiently good approximations to $f_D$ and $C_D$.  

\section{Preliminaries from computable analysis}

We first define a basis for the standard topology on $\XTC$ as follows.  Call an open rectangle $R$ \emph{rational} if the coordinates of its vertices are all rational numbers.
We first declare all rational rectangles to be basic.  Let $D_r(z)$ denote the open disk with radius $r$ and center $z$.  We then declare to be basic all sets of the form $\XTC - \overline{D_r(0)}$ where $r$ is a rational number.  
Let $\mathcal{O}(\XTC)$ be the set of all open subsets of $\XTC$.  Let $\mathcal{C}(\XTC)$ be the 
set of all closed subsets of $\XTC$.

We use the following naming systems.
\begin{enumerate}
	\item A name of a point $z \in \XTC$ is a list of all the basic neighborhoods that contain $z$.  
	
	\item A name of an open $U \subseteq \XTC$ is a list of all basic 
open sets whose closures are contained in $U$.   

	\item  A name of a closed $C \subseteq \XTC$ is a list of all basic open sets that intersect $C$.  
\end{enumerate} 

Suppose $f : \subseteq A \rightarrow B$ where $A, B$ are spaces for which we have established naming systems.  A Turing machine $M$ with a one-way output tape and access to an oracle $X$ is said to \emph{compute $f$} if
whenever a name of a point $x \in \dom(f)$ is written on the input tape, and $M$ is allowed to run indefinitely, a name of $f(x)$ is written on the output tape.  A \emph{name} of such a function then consists of a set $X \subseteq \N$ and a code of 
a Turing machine with one-way output tape which computes $f$ when given access to oracle $X$.  We note that a name of a function omits information about the domain of the function.
Thus, a name of a function may name many other functions as well.  This is an example of a 
\emph{multi-representation}.  This method of representing functions is taken from \cite{Grubba.Weihrauch.Xu.2008}.  Also, in order to compute a name of a function $f : \subseteq \XTC \rightarrow \XTC$ from some given data, it suffices to show that when a name of a $z \in \dom(f)$ is additionally provided, it is possible to uniformly compute a name of $f(z)$.  

Whenever we have established a naming system for a space, an object in that space is said to be \emph{computable} (with respect to the naming system), if it has a computable name.  

By an \emph{arc} we mean a subset of the plane that is the image of a continuous and injective map on $[0,1]$.  Such a map will be referred to as a \emph{parameterization} of the arc.  
Recall that a Jordan curve is a subset of the plane that is the image of a continuous map on $[0,1]$, $f$, with the property that $f(s) = f(t)$ only when $s,t \in \{0,1\}$.  Such a map will be referred to as a 
parameterization of the curve.

We will follow the usual mathematical practice of identifying an arcs and Jordan curves with their parameterizations.  However, when $\gamma$ is an arc or a Jordan curve, we speak of a name of $\gamma$ we mean a name of a parameterization of $\gamma$ rather than a name of $\gamma$ as a closed set.   
The distinction is necessary since, for example, J. Miller has shown that there is an arc that is computable as a closed set but that has no computable parameterization.  

When $\gamma$ is a Jordan curve, let $\Int(\gamma)$ denote its interior and let $\Ext(\gamma)$ denote its exterior.  The following 
follows from the main theorem of \cite{Gordon.Julian.Mines.Richman.1975}.

\begin{theorem}\label{thm:WINDING}
From a name of a Jordan curve $\gamma$, we can compute a name of the 
interior of $\gamma$ and a name of the exterior of $\gamma$.
\end{theorem}

A \emph{Jordan domain} is a finitely connected domain whose boundary components are Jordan curves.  
When we speak of a name of such a domain $D$ we mean an $n$-tuple $(P_1, \ldots, P_n)$ such that each $P_j$ is a name of a boundary component of $D$ and such that each boundary component of $D$ is named by at least one $P_j$.  When we speak of a name of a Jordan domain that is smooth (that is its boundary curves are continuously differentiable), we refer to a $2n$-tuple that includes in its name not only names of parameterizations of its boundary components, but also names of their derivatives.  It is necessary to use both since there are computable and continuously differentiable functions whose derivatives are not computable.   In this situtation, that is when we are naming a smooth Jordan domain $D$, we also require that each of these parameterizations be positively oriented with respect to $D$.   Intuitively, this means that as one travels around such a curve in the manner prescribed by its parameterization, the points of $D$ are always to the right.  This can be tested by means of the \emph{winding number}
\[
\eta(\gamma; z) = _{df} \frac{1}{2 \pi i} \int_\gamma \frac{1}{\zeta - z} d \zeta.
\]
If $\gamma$ is a boundary curve of $D$, and if $z_0$ is any point of the interior of $\gamma$, then $\gamma$ is positively oriented with respect to $D$ just in case 
\[
\eta(\gamma; z_0) = \left\{\begin{array}{cc}
								1 & \mbox{if $D$ is exterior to $\gamma$}\\
								-1 & \mbox{if $D$ is interior to $\gamma$}\\
								\end{array}
								\right.
\]

When $X \subseteq \C$, a \emph{ULAC function} for $X$ is a function $g : \N \rightarrow \N$ with the property that whenever $k \in \N$ and $z,w$ are distinct points of $X$ for which 
$|z - w| < 2^{-g(k)}$, there is an arc $A \subseteq X$ whose diameter is smaller than $2^{-k}$.

The following two theorems are proven in \cite{McNicholl.2011}.  Recall that when $X$ is a topological space and $a$ is a point of $X$, the \emph{connected component of $a$ in $X$} is the maximal connected subset of $X$ that contains $a$.

\begin{theorem}\label{thm:COMPONENT}
Suppose $D$ is an open disk, $A$ is an arc with ULAC function $g$, and $\zeta_0 \in A \cap D$.  Suppose $\zeta_1 \in D - A$ is such that $|\zeta_0 - \zeta_1| < 2^{-g(k)}$ where $k \in \N$ is such 
that $2^{-g(k)} + 2^{-k} \leq \max\{d(\zeta_0, \partial D), d(\zeta_1, \partial D)\}$.  Then, $\zeta_0$ is a boundary point of the connected component of $\zeta_1$ in $D - A$.  
\end{theorem}

We say that an arc $A$ \emph{links $z_0$ to $z_1$ via $U$} if $z_0$ and $z_1$ are the endpoints 
of $A$ and if $A \cap \partial U \subseteq \{z_0, z_1\}$.  

\begin{theorem}\label{thm:ACCESS}
From a name of an arc $A$, a point $z_0 \in \D - A$, and a name of a point $\zeta_0 \in A \cap \D$ that is a boundary point of the connected component of $z_0$ in $\D - A$, it is possible to uniformly compute a name of an arc $B$ that links $z_0$ to $\zeta_0$ via $\D - A$.
\end{theorem}

Let $f : [0,1] \rightarrow \C$ be a continuous function for which there exist numbers
\[
0 = t_0 < t_1 < \ldots < t_k = 1
\]
and points $v_0, v_1, \ldots, v_k \in \C$ such that 
\begin{eqnarray*}
f(x) & = & \frac{x - t_j}{t_{j+1} - t_j}(v_{j+1} - v_j) + v_j
\end{eqnarray*}
whenever $x \in [t_j, t_{j+1}]$.  $f$ is called a \emph{polygonal curve}.  The points $v_0, \ldots, v_k$ are called the \emph{vertices} of $f$.  We will call the points $v_1, \ldots, v_{k-1}$ the \emph{intermediate vertices} of $f$.  A \emph{rational polygonal curve} is a polygonal curve whose vertices are all rational.  The following is well-known.  

\begin{lemma}\label{lm:POLY.ARC}
From a name of a domain $U$ and names of distinct $p,q \in U$, it is possible to uniformly compute a polygonal arc $P$ from $p$ to $q$ that is contained in $U$ and whose intermediate vertices are rational. 
\end{lemma}

Whenever $\phi$ is a conformal map of $\D$ onto a bounded domain $U$, $\phi$ has a continuous extension to $\overline{\D}$ \cite{Pommerenke.1992}.  The next theorem follows from the main theorem of \cite{McNicholl.2011}.

\begin{theorem}\label{thm:BOUNDARY.EXTENSION}
From a name of a bounded domain $D$, a name of $\partial D$, a name of a conformal map $\phi$ of $\D$ onto $D$, and a ULAC function for $\partial D$, it is possible to uniformly compute a name of the continuous extension of $\phi$ to $\overline{\D}$.
\end{theorem}

When $D \subseteq \XTC - \{\infty\}$ is a domain, a real-valued function on $D$ is \emph{harmonic} if it has continuous first and second partial derivatives and satisfies the \emph{Laplace equation}
\[
\frac{\partial^2 u}{\partial x^2} + \frac{\partial^2 u}{\partial y^2} = 0.
\]
If $D$ is a domain that contains $\infty$, then a real-valued function $u$ on $D$ is harmonic if it is harmonic on $D - \{\infty\}$ and if there is a positive number $R$ such that $z \mapsto u(1/z)$ is harmonic on $D_R(0)$.

Suppose $D$ is a finitely-connected Jordan domain.  When $f$ is a bounded and piecewise continuous real-valued function on the boundary of $D$, the corresponding \emph{Dirichl\'et problem} is to find a harmonic function on $D$, $u$, such that 
\[
\lim_{z \rightarrow \zeta} u(z) = f(\zeta)
\]
at each $\zeta \in \partial D$ at which $f$ is continuous.  Such Dirichl\'et problems have solutions and that their solutions are unique \cite{Garnett.Marshall.2005}.  Accordingly, we say that the function $u$ is \emph{determined by the boundary data $f$} and denote it by $u_f$.

The following is proven in \cite{McNicholl.2010.1}.

\begin{lemma}\label{lm:DIRICHLET.DISK}
Given names of arcs $\gamma_1, \ldots, \gamma_n$ such that $\partial \D = \gamma_1 + \ldots + \gamma_n$,
and given names of continuous real-valued functions $f_1, \ldots, f_n$ such that $\gamma_j = \dom(f_j)$, we can compute a name of the harmonic
function $u$ on $\D$ defined by the boundary data
\[
f(\zeta) = \left\{ \begin{array}{cc}
					f_j(\zeta) & \zeta \in \gamma_j, \zeta \neq \gamma_j(0), \gamma_j(1)\\
					\max_j \max f_j & \mbox{otherwise}.\\
					\end{array}
					\right.
\]
In addition we can compute the extension of $u$ to $\overline{\D}$ except
at the endpoints of the arcs $\gamma_1, \ldots, \gamma_n$.
\end{lemma}

The following follows from the main result of \cite{Hertling.1999}.  See also \cite{Bishop.Bridges.1985} and \cite{Cheng.1973}.

\begin{theorem}\label{thm:1.CONNECTED}
From a name of a $1$-connected and non-degenerate domain $D$ that contains $\infty$ but not $0$, and a name of $\partial D$, it is possible to compute a names of $f_D$, $C_D$, and $\partial C_D$.
\end{theorem}

The following generalize results 
from \cite{Hertling.paper.0}.

\begin{theorem}[\bf Extended Computable Open Mapping Theorem]\label{thm:ECOMT}
From a name of a non-constant meromorphic $f$ and a name of an open subset of its domain, $U$, 
one can compute a name of $f[U]$.
\end{theorem}

\begin{proof}
From a name of $f$ and a name of an open $U \subseteq \dom(f)$, we can compute
a name of the restriction of $f$ to $U$.  It thus suffices to show that we can 
compute a name of $\ran(f)$.

Since the poles and zeros of a meromorphic function are isolated, for every $z \in \dom(f)$ there is a basic neighborhood of $z$ whose closure is contained in $\dom(f)$ and 
whose closure is either pole-free or zero-free.  Using the name of $f$, we can build a list
of the basic neighborhoods whose closures are zero-free and contained in 
$\dom(f)$.  We can also build a list of the basic neighborhoods whose closures are pole-free and contained in 
$\dom(f)$.  We scan these lists as we build them, and do the following.  Suppose 
$V$ is a pole-free neighborhood whose closure is contained in $\dom(f)$.
We can then apply the Computable Open Mapping Theorem of Hertling
\cite{Hertling.paper.0} and begin listing all finite basic neighborhoods 
whose closures are contained in $f[V]$ as we go along.  Suppose 
$V$ is zero-free.  Again, using Hertling's Computable Open Mapping Theorem, 
we can begin listing all finite basic neighborhoods whose closures are 
contained in $\frac{1}{f}[V]$.  We can then also list all basic neighborhoods whose
closures are contained in the image of $\frac{1}{z}$ on this set.
We can work these neighborhoods into our output list as we go along.

What we will produce by this process is a list of basic neighborhoods
$V_0, V_1, \ldots$ such that 
\[
\ran(f) = \bigcup_j \overline{V_j}.
\]
However, it may be the case that not every basic neighborhood $V$ with 
$\overline{V} \subseteq \ran(f)$ will appear in this list.  What we have formed here
so far is known as an \emph{incomplete name}.  However, it is quite easy to remedy 
the situation.  Whenever basic neighborhoods $U_1, \ldots, U_k$ are listed, we begin working into our list all 
basic neighborhoods contained in $\bigcup_j U_j$.  It follows from the compactness
of $\XTC$ that the resulting list is complete.
\end{proof}

%

\begin{theorem}[\bf Extended Computable Closed Mapping Theorem]
From a name of a meromorphic $f$ and a name of a closed set contained in its domain, $C$,
one can compute a name of $f[C]$.
\end{theorem}

\begin{proof}
Begin scanning the names of $f$ and $C$ simultaneously.  Suppose we discover a pair 
$(V,U)$ such that $f[V] \subseteq U$ and such that $V \cap C \neq \emptyset$.  We can 
then list $U$ as a basic neighborhood that hits $f[C]$.  At the same time, 
we work into our list all basic neighborhoods that contain $U$. 
It follows that the resulting list is a name of $f[C]$.
\end{proof}

\section{Computing the sequences in the Koebe construction}\label{sec:KOEBE}

Let $D$ be a non-degenerate $n$-connected domain that contains $\infty$ but not $0$.  
We inductively define sequences 
$\{D_k\}_{k = 0}^\infty$, $\{D_{k,1}\}_{k = 0}^\infty$, $\ldots$, $\{D_{k,n}\}_{k = 0}^\infty$, and $\{f_k\}_{k = 0}^\infty$ as follows.  

To begin, let $D_{0,1}, \ldots, D_{0,n}$ be the connected components of $\XTC - D$.    
Let $D_0 = D$.  Let $f_0 = Id_{D}$.

Let $k \in \N$, and suppose $f_k, D_k, D_{k, 1}, \ldots, D_{k,n}$ have been defined. 
Let $k' \in \{1, \ldots, n\}$ be equivalent to $k+1$ modulo $n$.  
Let $f_{k+1}$ be the conformal map of $\XTC - D_{k, k'}$ onto 
a circular domain $C$ such that $f_{k+1}(z) = z + O(z^{-1})$.  
Now, let $D_{k+1} = f_{k+1}[D_k]$.  Let $D_{k+1, j} = f_{k+1}[D_{k,j}]$ when 
$j \neq k'$.  Let $D_{k+1, k'} = \XTC - C$.  

Let $g_k = f_k \circ \ldots \circ f_0$.  

%
 %
%
%
%
%
 %
%
In \cite{Koebe.1912}, P. Koebe outlines a proof of the following.

\begin{theorem}
The sequence $\{g_k\}_{k \in \N}$ converges uniformly to $f_D$.
\end{theorem}

Bounds on the rate of convergence of $\{g_k\}_{k \in \N}$ were computed by 
P. Henrici \cite{Henrici.1986} and D. Gaier \cite{Gaier.1959}.  However, these bounds
are stated in terms of quantities associated with the circular domain $C_D$.  In \cite {Andreev.McNicholl.2011}, 
Andreev and McNicholl give bounds which are stated solely in terms of 
quantities associated with the domain $D$.  The results in this paper are based on these bounds.

We now show that the sequences $\{D_k\}_k$, $\{f_k\}_k$, and $\{\partial D_{k,j}\}_{k,j}$ 
generated from the Koebe construction can be computed from the initial data 
$D, \partial D, n$.  Our first task is to show that from these initial data we can compute
the boundary components $\partial D_{0,1}, \ldots, \partial D_{0,n}$.  
At first sight, this may seem obvious as it may seem that we can simply look at the boundary of $D$ and determine the component boundaries.  However, our initial data
do not specify the entire boundary of $D$ (which may not even be given by Jordan curves) all at once; they merely give us a sequence
of approximations to the boundary, and we must sort these into approximations to the individual components.  We also need to show that we can ``cover-up" components using these initial data.  Mathematically, this means computing the complements 
of the individual components of the complement of $D$. 

\begin{theorem}\label{thm:DECOMP}
From a name of a finitely connected domain $D$ that contains $\infty$ but not $0$, a name of $\partial D$, and the number of connected components of $\XTC - D$, it is possible to uniformly compute names of the complements of $D_{0,1}$, $\ldots$, $D_{0, n}$ as well as names of their boundaries.
\end{theorem}

\begin{proof}
Let $n$ denote the number of boundary components of $D$.  There is a sequence of $n$-connected and unbounded Jordan domains $\Omega_1, \Omega_2, \ldots$ such that 
\begin{eqnarray*}
\overline{\Omega_j} & \subseteq & D\\
\overline{\Omega_{j+1}} & \subseteq & \Omega_j\\
D & = & \bigcup_j \Omega_j
\end{eqnarray*}
It follows that there are Jordan curves $\gamma_1, \ldots, \gamma_n$ such that for each $j \in \{1, \ldots, n\}$ 
\begin{eqnarray*}
D_{0,j} & \subseteq & \Int(\gamma_j)\\
\gamma_j & \subseteq & \bigcap_{k \neq j} \Ext(\gamma_k).
\end{eqnarray*}
Every Jordan curve can be approximated with arbitrary precision by a rational polygonal Jordan curve.  Hence, we can additionally assume that each $\gamma_j$ is a rational polygonal curve.

So, we first search for rational polygonal Jordan curves $\gamma_1$, $\ldots$, $\gamma_n$ and rational rectangles $S_1, \ldots, S_n$ such that 
\begin{eqnarray*}
\gamma_j & \subseteq & \bigcap_{k \neq j} \Ext(\gamma_k)\\
\overline{S_j} & \subseteq & \Int(\gamma_j)\mbox{, and}\\
\emptyset & \neq & S_j \cap \partial D
\end{eqnarray*}
It follows that each $\gamma_j$ contains in its interior exactly one connected component of $\XTC - D$; label this component $D_{0, j}$.

We then list a basic open set $S$ as one that intersects $\partial D_{j,0}$ just in case there is a rational rectangle $R$ such that $\overline{R}  \subseteq  S \cap \Int(\gamma_j)$ and $R \cap \partial D \neq \emptyset$.

We now compute a name of $\XTC - D_{0, j_0}$.  If $z \in \XTC - D_{0,j}$, then there is an $\Omega_k$ one of whose boundary curves has the property that $z$ belongs to its exterior and also that $D_{0,j}$ is contained in its interior.  It follows that there is a basic open set $S$ and a rational polygonal Jordan curve $P$ such that $z \in S \subseteq \overline{S} \subseteq \Ext(P)$, and such that $\partial D_{0,j} \subseteq \Int(P)$.

So, we list a basic set $S$ as one whose closure is contained in the complement of $D_{0,j}$ if we find a rational polygonal curve $P$ and rational rectangles $R_1, \ldots, R_m$ such that 
$\overline{S} \subseteq \Ext(P)$ and such that 
\[
\partial D_{0,j} \subseteq \bigcup R_k \subseteq \bigcup \overline{R_k} \subseteq \Int(P).
\]
\end{proof}

The following now follows from Theorem \ref{thm:DECOMP}, Theorem \ref{thm:1.CONNECTED}, and primitive recursion (see, \emph{e.g.}, Theorem  2.1.14 of \cite{Weihrauch.2000}).

\begin{theorem}\label{thm:KOEBE.COMP}
From a name of a finitely connected and non-degenerate domain $D$ that contains $\infty$ but not $0$, a name of $\partial D$,  the number of connected components of $\XTC - D$, a $k \in \N$, and a $j \in \{1, \ldots, n\}$, it is possible to uniformly compute names of $f_k$, $D_k$, and $\partial D_{k,j}$.
\end{theorem}

\section{Bounding the error in the Koebe construction}\label{sec:BOUNDING}\label{sec:BOUND}

In this section we summarize the results of Andreev and McNicholl on the error in the Koebe construction \cite{Andreev.McNicholl.2011}.  To this end, we first review some harmonic function theory.

Suppose $u,v$ are harmonic functions on a domain $D$ and 
\begin{eqnarray}
\frac{\partial u}{\partial x} & = & \frac{\partial v}{\partial y}\label{eqn:CR.1}\\
\frac{\partial u}{\partial y} & = & -\frac{\partial v}{\partial x}\label{eqn:CR.2}
\end{eqnarray}
Then, $v$ is said to be a \emph{harmonic conjugate} of $u$.  Equations 
(\ref{eqn:CR.1}) and (\ref{eqn:CR.2}) are known as the \emph{Cauchy-Riemann equations}.
It is well-known that $v$ is a harmonic conjugate of $u$ if and only if 
$u + iv$ is analytic.  If $D$ is simply connected, then it follows that $u$ has a harmonic conjugate and that all of its harmonic conjugates differ by a constant.  However, if $D$ is multiply connected, then $u$ may not have a harmonic conjugate \cite{Nehari.1952}.

Suppose $u$ is harmonic on a domain $D$ and that $\gamma$ is a smooth arc or Jordan curve that is contained in $D$.  The \emph{normal derivative} of $u$ is denoted $\frac{\partial u}{\partial \normal}$ and is 
defined by the equation
\[
\frac{\partial u}{\partial \normal}(t) =   \left(\frac{\partial u}{\partial x}(\gamma(t)) \gamma_2'(t) - \frac{\partial u}{\partial y}(\gamma(t)) \gamma_1'(t)\right) \frac{1}{|\gamma'(t)|}
\]
Intuitively, the normal derivative of $u$ is the rate of change of $u$ in the direction of the vector that is perpendicular to the tangent vector of $\gamma$ and is to the right as one traverses $\gamma$ in the 
direction of the given parameterization.
We also define
\[
\frac{\partial u}{\partial s} = \left(\frac{\partial u}{\partial x} x'(t) + \frac{\partial u}{\partial y} y'(t) \right)  \frac{1}{|x'(t) + i y'(t)|}.
\]
The Cauchy-Riemann equations can now be rewritten as
\begin{eqnarray*}
\frac{\partial u}{\partial \normal} & = & \frac{\partial v}{\partial s}\\
\frac{\partial v}{\partial \normal} & = & -\frac{\partial u}{\partial s}.
\end{eqnarray*}
We also let $ds$ denote the differential of arc length.  That is, $ds = |\gamma'(t)| dt$.  

Suppose $D$ is a Jordan domain, and let $\Gamma_1, \ldots, \Gamma_n$ denote its boundary curves.  For each $j$, let $\omega(\cdot, \Gamma_j, D)$ be the harmonic function on $D$ determined by the boundary data
\[
f(\zeta) = \left\{ \begin{array}{cc}
					1 & \zeta \in \Gamma_j\\
					0 & \zeta \not \in \Gamma_j\\
					\end{array}
					\right.
\]
This harmonic function is referred to as the \emph{harmonic measure function} of $\Gamma_j$.  

We then let $P_{k,j}^D$ denote the conjugate period of $\omega(\cdot, \Gamma_j, D)$ around $\Gamma_k$.  The matrix $(P_{k,j})_{k,j}$ is known as the \emph{Riemann matrix of $D$} and plays a role in the solution of many problems in harmonic and analytic function theory.  See \emph{e.g.}, \cite{Axler.Bourdon.Wade.2001}, \cite{Schiffer.1950}, \cite{Thurman.1994}.

We can now give the following definition from \cite{Andreev.McNicholl.2011}.

\begin{definition}\label{def:ER}
When $D$ is an unbounded Jordan domain whose boundary curves are denotes $\Gamma_1, \ldots, \Gamma_n$, we let 
\[
\mathcal{E}_R(D) = \min_j \exp\left\{ - \left( \frac{P_{j,j}^D}{\omega(\infty, \Gamma_j, D)^2} + 1\right) R^2 \right\}.
\]
\end{definition}

Fix $j \in \{1, \ldots, n\}$.  We define a number $\lambda_j(D)$ as follows.  Let 
$g$ be a conformal map of the exterior of $\Gamma_j$ onto the interior of $\Gamma_j$.   
Let $E_1 = g[D]$.  Let $E_2$ be the unbounded domain whose boundary components are 
$g[\Gamma_k]$ where $k \in \{1, \ldots, n\} - \{j\}$.  Hence, $E_2$ is an $(n-1)$-connected domain.  Let $E_3 = f_{E_2}[E_1]$.  Therefore, $E_3$ is bounded by the curves
$f_{E_2}[\partial \Gamma_j]$, $f_{E_2}g[\Gamma_1], \ldots, f_{E_2}g[\Gamma_{j-1}], f_{E_2}g[\Gamma_{j+1}], \ldots, f_{E_2}g[\Gamma_n]$.  The last $n-1$ of these curves are circles.  
Set 
\[
M = 2 \max\{ |z|\ :\ z \in \overline{E_3}\}.
\]
Thus, $E_4 =_{df} \frac{1}{M} E_3 \subset \D$.  Let $\mathcal{C}_1, \ldots, \mathcal{C}_{n-1}$ label the circles
\[
\frac{1}{M}f_{E_2}g[\Gamma_1], \ldots, \frac{1}{M}f_{E_2}g[\Gamma_{j-1}], \frac{1}{M}f_{E_2}g[\Gamma_{j+1}], \ldots, \frac{1}{M}f_{E_2}g[\Gamma_n].
\]
We now arrive at the definition of $\lambda_j(D)$.

\begin{definition}\label{def:LAMBDAj}
When $\mathcal{C}_1, \ldots, \mathcal{C}_{n-1}$ are obtained as in the above process, we let
\begin{eqnarray*}
\lambda_j(D) & = & \min_{k_1 \neq k_2} \rho(\mathcal{C}_{k_1}, \mathcal{C}_{k_2}).
\end{eqnarray*}
\end{definition}

The following is proven in \cite{Andreev.McNicholl.2011}.

\begin{theorem}\label{thm:ERROR.KOEBE}
Suppose $D_{2/R}(z_0) \subseteq \XTC - D \subseteq D_{R/2}(0)$.  Let $\mathcal{E}_R = \mathcal{E}_R(D)$.  Then, 
for all $j \in \N$, $|g_j - f_D| < \gamma^+_D (\mu_D^+)^{\lfloor j/n \rfloor}$ where 
\begin{eqnarray*}
\gamma_D^+ & = & 72 R \left[ \frac{1}{\mathcal{E}_R^2 \log( 1+ \frac{1}{4} \mathcal{E}_R^3 \min\{\lambda_1, \lambda_2\})} + 1\right] \frac{\mathcal{E}_R^2 \min\{\lambda_1, \lambda_2\} + 1}{\mathcal{E}_R^3 \min\{\lambda_1, \lambda_2\}}\\
\mu_D^+ & = & \frac{1}{1 + \frac{1}{4} \mathcal{E}_R^3 \min\{\lambda_1, \lambda_2\} }
\end{eqnarray*}
\end{theorem}

\section{Some computations with harmonic functions}\label{sec:HARMONIC}

In order to apply the results of Section \ref{sec:BOUNDING}, it is necessary to first demonstrate the computability of certain fundamental operations on harmonic functions.  We begin with the computation of local conjugates.

\begin{theorem}[\bf Computable conjugation]\label{thm:CONJUGATE}
From a name of a harmonic function, $u$, and names of $z_0, R$ such that 
$\overline{D_R(z_0)} \subseteq \dom(u)$, it is possible to uniformly compute a name of a local 
harmonic conjugate of $u$ with domain $D_r(z_0)$.
\end{theorem}

\begin{proof}
We first translate $z_0$ to the origin.  Allow $u$ to also denote the resulting function.
Given $z \in D_R(0)$, we first compute $r$ such that $|z| < r < R$.  We can then compute
\[
\tilde{u}(z) = \int_0^{2\pi} u(re^{i\theta}) \frac{re^{-i \theta}z - re^{i \theta}\overline{z}}{|re^{i\theta} - z|^2} d\theta.
\]
Thus, $\tilde{u}$ is the harmonic conjugate of $u$ on $D_R(0)$ that vanishes at $0$.  (See, for example, page 178 of \cite{Gilman.Kra.Rodriguez.2007}.)  We now translate $0$ back to $z_0$.  
Allow $\tilde{u}$ to also denote the resulting function.  

We have shown that from names of $u, R, z_0, z$, we can compute 
$\tilde{u}(z)$. 
\end{proof}

By means of local conjugation, we can now show that differentiation of harmonic functions is a computable operator.

\begin{theorem}\label{thm:DIFFERENTIATION.HARMONIC}
From a name of a harmonic function, $u$, we may compute a name of $u'|_\C$.
\end{theorem}

\begin{proof}
Given names of $u$ and $z \in \dom(u) \cap \C$, we first read these names until we find
a rational rectangle $R$ that contains $z$ and whose closure is contained in $\dom(u)$.  This then allows us to compute the radius of a closed disk $D$ centered at $z$ that is contained in $R$.  By Theorem \ref{thm:CONJUGATE}, we can compute a harmonic 
conjugate of $u$ on the interior of $D$, $\tilde{u}$.  Let $f = u + i\tilde{u}$.  We can then 
compute $f'$.   By an elementary computation, $\frac{\partial u}{\partial x} = Re(f')$ and 
$\frac{\partial u}{\partial y} = Re(i f')$. 
\end{proof}

We now show that the operation of harmonic extension, which is used to expand the domain of a harmonic function, is computable.

\begin{theorem}[\bf Computable Harmonic Extension]\label{thm:HARMONIC.EXTENSION}
Given a name of a domain $D$, a name of a harmonic $u : \overline{D} \rightarrow \R$, and names of conformal $f_1, \ldots, f_n$ such that
\begin{itemize}
	\item $\partial \D \subseteq \dom(f_j)$, 
	\item $\gamma_j =_{df} f_j[\partial \D]$ is a boundary component of $D$ on which $u$ is zero, and
	\item $\gamma_1, \ldots, \gamma_n$ are distinct,
\end{itemize}
we can compute a neighborhood of $D \cup (\bigcup_j \gamma_j)$, $D'$, and a harmonic extension of $u|_D$ to $D'$.
\end{theorem}

\begin{proof}
Fix $j \in \{1, \ldots, n\}$ for the moment.  From the name of $f_j$, we can 
compute a name of $\dom(f_j)$.  We can then compute a name of $\exp^{-1}[\dom(f_j)]$.  (See \emph{e.g.} Theorem 6.2.4.1 of \cite{Weihrauch.2000}.)
We can then compute a covering of the line segment from 
$-i\pi$ to $i\pi$ by rational rectangles $R_1, \ldots, R_k$ whose closures 
are contained in $\exp^{-1}[\dom(f_j)]$.  Let $r$ be the minimum distance 
from a vertical side of one of these rectangles to the $y$-axis.  Hence, $r$
is computable from $R_1, \ldots, R_k$.  It now follows that for each $-\pi \leq y \leq \pi$, each 
point on the line segment from $-r + iy$ to $r + iy$ is contained in 
$\exp^{-1}[\dom(f_j)]$.  We can now compute
a positive rational number $r_{0,j}$ such that $r_{0,j} < e^{-r}$.    
Let $A_j$ be the annulus centered at the origin and with inner radius $1 - r_{0,j}$ and 
outer radius $1 + r_{0,j}$.  Hence, $\overline{A_j} \subseteq \dom(f_j)$.   
Let $g$ be the reflection map for $\partial \D$.  \emph{i.e.}
\[
g(z) = \frac{1}{\overline{z}}.
\] 
Hence, $A_j$ is closed under reflection.   

Let
\[
D' = D\ \cup\ \bigcup_{j =1 }^n f_j[A_j].
\]
We can choose $r_{0,1}, \ldots, r_{0,n}$ so that  $f_1[A_1], \ldots, f_n[A_n]$ are 
pairwise disjoint.  
It follows from the Extended Computable Open Mapping Theorem (Theorem \ref{thm:ECOMT}) that $D'$ can be computed from the given data.  We define 
$v$ on $D'$ as follows.  Given $z \in D'$, if $z \in D$, then let 
$v(z) = u(z)$.  Otherwise, there exists unique $j$ such that 
$z \in f_j[A_j]$, and we let
\begin{eqnarray}
v(z) & = & uf_jgf_j^{-1}(z).\label{eqn:REFLECTION}
\end{eqnarray}  
It follows from Schwarz Reflection (see \emph{e.g.} Theorem 4.12 of \cite{Axler.Bourdon.Wade.2001}), that $v$ is harmonic on $D'$.
It only remains to show we can compute $v$ from the given data.
Here is how we do this.  Given a name $p$ of a $z \in D'$, we read 
$p$ until we find a subbasic neighborhood $R$ such that 
either $\overline{R} \subseteq D$ or 
$\overline{R} \subseteq f_j[A_j]$.  In the first case, we simply compute $u(z)$.  
In the second case, we can use (\ref{eqn:REFLECTION}).  
\end{proof}

In \cite{McNicholl.2011.a}, an iterative method for the solution of Dirichl\'et problems for Jordan domains is developed along with a closed form for their solution.  Based on these results and the results in \cite{McNicholl.2011.b} on computing boundary extensions of conformal maps (see also \cite{McNicholl.2010.1}), we now show that solving Dirichl\'et problems for Jordan domains is a computable operation.

\begin{theorem}[\bf Computable Solution of Dirichl\'et Problems]\label{thm:DIRICHLET}
Given a name of a bounded Jordan domain $D$ and a name of a continuous 
$f : \partial D \rightarrow \R$, 
it is possible to uniformly compute a name of $u_f$.  Furthermore, we can compute the continuous extension of 
this solution to $\overline{D}$.
\end{theorem}

\begin{proof}
\newcommand{\Harm}{\mbox{\it Bharm}}
Let $\gamma_1, \ldots, \gamma_n$ denote the boundary curves of $D$ with $\gamma_1$ 
being the outermost.

Let $\sigma_1, \ldots, \sigma_{n-1}$, $\tau_1, \ldots, \tau_{n-1}$ be pairwise disjoint arcs such that 
$D_1 =_{df} D - \bigcup \tau_j$ and $D_2 =_{df} D - \bigcup \sigma_j$ are simply connected.  The case when $D$ is bounded by four curves is illustrated in Figure \ref{fig:ARCS}.
\begin{center}
\begin{figure}[!h]
\resizebox{4.5in}{4.5in}{\includegraphics{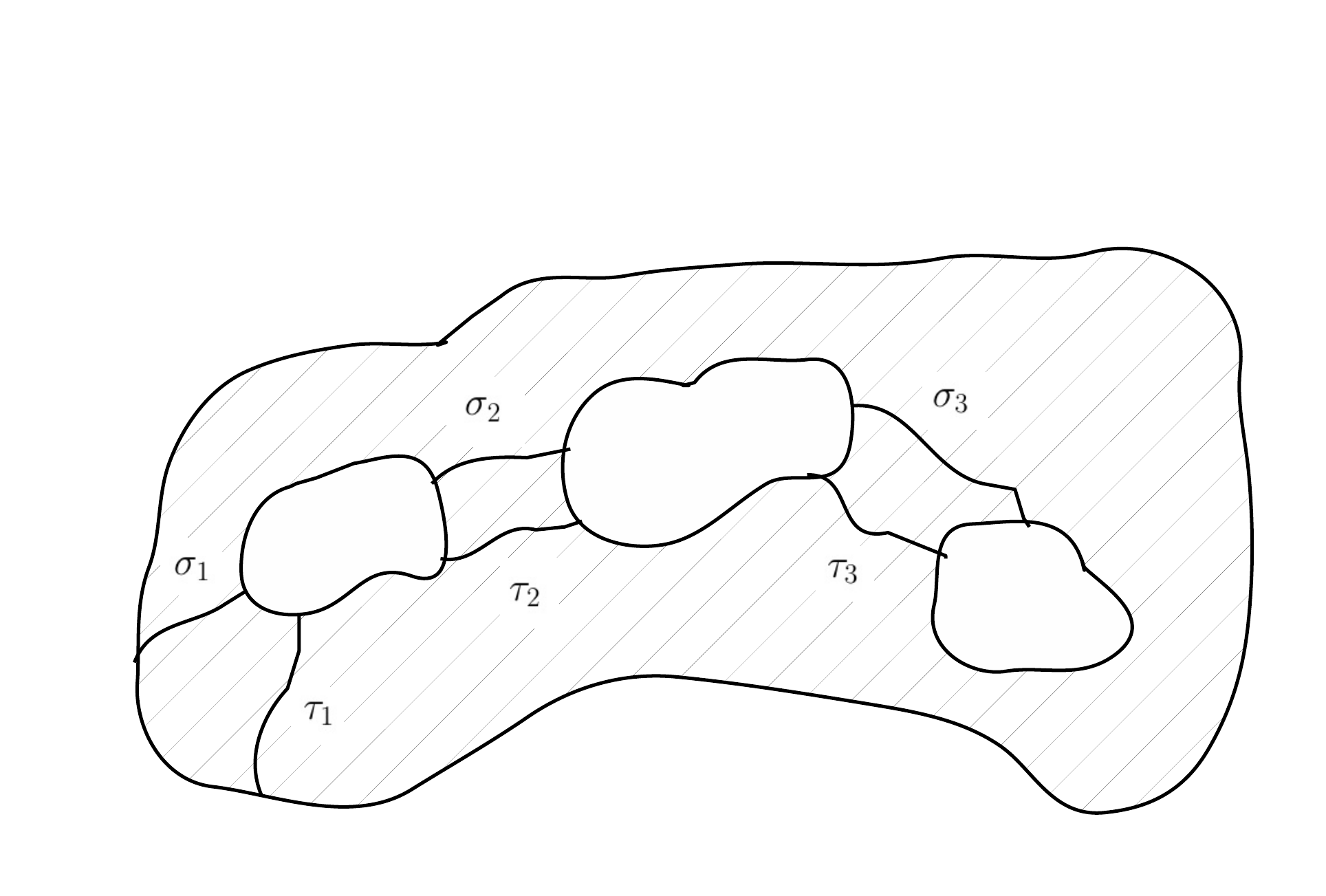}}
\caption{ }\label{fig:ARCS}
\end{figure}
\end{center}
More to the point, $\sigma_j$ and $\tau_j$ link $\gamma_j$ and $\gamma_{j+1}$ via $D$.  We also choose these arcs so that $z \not \in \sigma_j, \tau_j$.  Such arcs can be computed by repeated application of Theorem \ref{thm:COMPONENT}, Theorem \ref{thm:ACCESS}, and Lemma \ref{lm:POLY.ARC}.
Let $\sigma = \bigcup_j \sigma_j$, and let $\tau = \bigcup_j \tau_j$.

Let $\phi_j$ be a continuous map of $\overline{\D}$ onto $\overline{D_j}$ that is conformal on $\D$.  Thus, each point of $\overline{D} - \sigma$ has exactly one preimage under $\phi_1$, and each point of $\overline{D} - \tau$ has exactly one preimage under $\phi_2$.  The existence of these maps follows from Theorem 2.1 of \cite{Pommerenke.1992}.  It follows from Theorem \ref{thm:BOUNDARY.EXTENSION} that they can be computed from the given data.
Let: 
\begin{eqnarray*}
A_j & = & \phi_j^{-1}[\partial D]\\
B_1 &= & \phi_1^{-1}[\sigma]\\
B_2 & = & \phi_2^{-1}[\tau]
\end{eqnarray*}
Hence, $A_j$, $B_j$ can be computed from the given data.

Let $H_1$ be the harmonic function on $D_1$ determined by the boundary conditions 
\[
H_1(\zeta) = \left\{\begin{array}{cc}
						f(\zeta) & \mbox{if $\zeta \in \partial D$}\\
						0 & \mbox{if $\zeta \in \sigma - \partial D$}\\
						\end{array}
						\right. ,
\]
and let $H_2$ be the harmonic function on $D_2$ determined by the boundary conditions
\[
H_2(\zeta) = \left\{\begin{array}{cc}
						f(\zeta) & \mbox{if $\zeta \in \partial D$}\\
						0 & \mbox{if $\zeta \in \tau - \partial D$}\\
						\end{array}
						\right.
\]
Let $H_3$ be the harmonic function on $D_1$ determined by the boundary conditions
\[
H_3(\zeta) = \left\{\begin{array}{cc}
					H_2(\zeta) & \mbox{if $\zeta \in \sigma$}\\
					0 & \mbox{otherwise}\\
					\end{array}
					\right.
\]
Finally, let 
\begin{eqnarray}
h & = & H_1 + H_3. 
\end{eqnarray}
It follows from Theorem \ref{thm:BOUNDARY.EXTENSION} and Lemma \ref{lm:DIRICHLET.DISK} that $H_1$, $H_2$, $H_3$, and thus $h$ can 
be computed from the given data.  

For each $\zeta_1 \in B_2$, let $K(\cdot, \zeta_1)$ be the harmonic function on $D_1$ defined by the boundary conditions
\[
K(\zeta, \zeta_1) = \left\{ \begin{array}{cc}
								0 &  \zeta \not \in \sigma\\
								2\pi P(\phi_2^{-1}(\zeta), \zeta_1) & \zeta \in \sigma\\
								\end{array}
								\right.
\]
So, $K$ can also be computed from the given data.

Note that $\phi_2^{-1}(\zeta)$ is bounded away from $B_2$ as $\zeta$ ranges over $\sigma$.  Accordingly, let 
\begin{eqnarray*}
m & = & \max\left\{ \int_{B_2} P(\phi_2^{-1}(\zeta), \zeta_1) ds_{\zeta_1}\ :\ \zeta \in \sigma\right\}.
\end{eqnarray*} 
It follows that $m < 2\pi$ and that $m$ can be computed from the given data.

Let $\Harm(D_1)$ denote the space of bounded harmonic functions on $D$ with the sup norm. $\Harm(D_1)$ is complete.

We now define an operator on $\Harm(D_1)$.  When $v$ is harmonic on $D_1$, let $F(v)$ denote the function on $D_1$ defined by the equation
\[
F(v)(z) = h(z) + \frac{1}{(2\pi)^2} \int_{B_2} K(z, \zeta_1) v(\phi_2(\zeta_1)) ds_{\zeta_1}.
\]

The following is proven in \cite{McNicholl.2011.a}.

\begin{lemma}\label{lm:CONTRACTION}
$F$ is a contraction map on $\Harm(D_1)$.  In particular, for all $v_1, v_2 \in \Harm(D_1)$, 
\[
\parallel F(v_1) - F(v_2) \parallel_\infty \leq \frac{m}{2\pi} \parallel v_1 - v_2 \parallel_\infty.
\]
\end{lemma}

So, let $u$ be the fixed point of $F$.  It is shown in \cite{McNicholl.2011.a} 
that $u$ is the restriction of $u_f$ to $D_1$.  

\newcommand{\zero}{\mathbf 0}
Let $\zero$ denote the zero function on $D_1$.   Hence, 
$u_f(z) = \lim_{t \rightarrow \infty} F^t(\zero)(z)$.  Furthermore, we can compute a modulus 
of convergence for $\{F^t(\zero)(z)\}_t$.  Thus, $u_f(z)$ can be computed from the given data.

\end{proof}

We now obtain the following as immediate corollaries of Theorems \ref{thm:DIRICHLET} and \ref{thm:DIFFERENTIATION.HARMONIC}.

\begin{corollary}[\bf Computability of Harmonic Measure]\label{cor:OMEGA}
From a name of a Jordan domain $D$  it is possible to uniformly compute names of the corresponding harmonic measure functions.  Furthermore, we can compute their continuous extensions to $\overline{D}$.
\end{corollary}

\begin{corollary}[\bf Computability of the Riemann Matrix]
Given the same initial data as in Theorem \ref{cor:OMEGA}, it is possible to uniformly compute a name of each $P_{j,k}^D$.
\end{corollary}

\begin{corollary}\label{cor:ER.COMP}
From a name of an unbounded Jordan domain $D$, and a name of an $R > 0$, we can compute a name of $\mathcal{E}_R(D)$.
\end{corollary}

\section{Computably bounding the error in the Koebe construction}

The error bounds in Section \ref{sec:BOUND} rely on two quantities, $\mathcal{E}_R(D)$ and $\lambda_j(D)$.  The computability of the former has now been demonstrated.  To compute $\lambda_j(D)$, at least if one wishes to compute it by following its definition, requires the computability of conformal mapping onto circular domains for $(n-1)$-connected domains.  We are thus led to consider an inductive procedure.  The bounds in Section \ref{sec:BOUND} only apply when $D$ has at least three boundary components.  Thus, the starting point for this induction is the 2-connected case which leads us to the following theorem.

\begin{theorem}\label{thm:2.CONNECTED}
From a name of a 2-connected domain $D$ that contains $\infty$ but not $0$, and a name of $\partial D$, it is possible to uniformly compute names of $f_D$ and $\partial C_D$.  
\end{theorem}

\begin{proof}
There is a unique number $\mu$ such that $D$ is conformally equivalent to the annulus
\[
A =_{df} \{z \in \C\ :\ \mu^{-1} < |z| < 1\}.
\]
We use the \emph{Komatu construction} to conformally maps $D$ onto $A$ \cite{Henrici.1986}.   Again, we refer to the notation in Section \ref{sec:INTRO}.  Let $z_0$ and $r_0$ denote the center and radius respectively of $D_{2,2}$.  These can be computed from the given data.  Let $E$, $E_1$, and $E_2$ denote the image of $D_2$, $\partial D_{2,1}$, and $\partial D_{2,2}$ respectively under the map
\[
z \mapsto \frac{r_0}{z - z_0}.
\]
Hence, $E_2 = \partial \D$.  Let $\alpha$ be a conformal automorphism of the unit disk that exchanges $z_0$ and $0$.  Let $F_0$, $F_{0,1}$, and $F_{0,2}$ denote the images of $E$, $E_1$, and $E_2$ respectively under $\alpha$.  (Of course, $E_2 = \partial \D$.)  These can be computed from the given data.

Let $h_1$ be the conformal map of the exterior of $F_{0,1}$ onto the exterior of 
$\partial \D$.  Let $F_1$ be the image of $h_1$ on $F_0$, and let 
$F_{1,1}$ and $F_{1,2}$ denote the corresponding boundary curves.
Let $h_2$ be the conformal map of the interior of $F_{1,2}$ onto 
$\D$ such that $h_2(0) = 0$.  Call the image of $h_2$ on $F_1$ $F_2$ and 
let $F_{2,1}$ and $F_{2,3}$ be the corresponding boundary curves.  We now repeat
this process and obtain maps $h_3, h_4, \ldots$.

In \cite{Henrici.1986}, Theorem 9.4, it is shown that the sequence $\{h_k\}_{k \in \N}$ converges to a conformal map $h$ of $F_0$ onto $A$ and that 
\begin{equation}
|h_k(z) - h(z)| \leq 13 (\mu^{-1})^{2k}.\label{eqn:RING.EST}
\end{equation}

So, to compute $h$ from the given data, it suffices to compute a number between $\mu^{-1}$ and $1$ from the given data.  This is accomplished as follows.
Let $D[]$ denote the Dirichl\'et integral operator:
\[
D[f] = \int \int_{\overline{E}} \left[ \left( \frac{\partial f}{\partial x} \right)^2 + \left( \frac{\partial f}{\partial y} \right)^2 \right] dA
\]
where $dA$ is the differential of area.  Let $\phi(z) = \omega(z, E_2, E)$.  In \cite{Gaier.1974}, it is shown that $\mu = \exp(2\pi/ D[\phi])$.  It follows from the results in Section \ref{sec:HARMONIC}, that we can compute from the given data a number $N$ such that $0 < N < D[\phi]$.  
Thus, $\mu^{-1} < \exp(-2\pi/ N) < 1$.  

Set $z_2 = h(z_1)$.  Compute $R$ such that $\overline{D_R(z_2)} \subseteq A$.  Let $U$ denote the image of $A$ under
the map 
\[
z \mapsto \frac{R^2}{z - z_2}.
\]
It follows that $U$ is an unbounded circular domain that contains $\infty$ but not $0$.  Let 
\[
f(z) = \frac{R^2}{h\alpha\left( \frac{r_0}{f_2(z) - z_0}\right) - z_2}.
\]
Thus, $f$ maps $D$ onto $U$.  In addition, $f(\infty) = \infty$.  By normalizing, we obtain $f_D$.

It follows from Theorem \ref{thm:BOUNDARY.EXTENSION} that we can compute the continuous extension of $h_{j-1}$ to $\overline{F_j}$.  By continuity, the estimate \ref{eqn:RING.EST} holds for these extensions as well.  Thus, if $D$ is Jordan, and if we are also provided with names of the boundary curves of $D$, then we can compute the continuous extension of $f_D$ to $\overline{D}$.  
\end{proof}

By primitive recursion, we now obtain the following.

\begin{theorem}\label{thm:n.CONNECTED}
From a name of a finitely connected and non-degenerate domain $D$ that contains $\infty$, a name of $\partial D$, and the number of boundary components of $D$, it is possible to uniformly compute $f_D$ and $\partial C_D$.
\end{theorem}

\section{Boundary extensions}

By applying Theorem \ref{thm:BOUNDARY.EXTENSION} to each iteration of the Koebe construction, we obtain the following.

\begin{corollary}\label{thm:BOUNDARY.EXT}
From a name of a finitely connected and non-degenerate domain $D$ with locally connected boundary, and the number of its boundary components, a name of $\partial D$, and a CIK function for $\partial D$, it is possible to compute a name of the continuous extension of 
$f^{-1}_D$ to $\overline{C_D}$.
\end{corollary}

\section{A reversal}

We have now demonstrated that in order to compute $f_D$ and $C_D$, it is sufficient to have a name of $D$, a name of its boundary, and the number of its boundary components.  We now show that a name of the boundary is not only sufficient but necessary.  We will use the following lemma from Hertling \cite{Hertling.paper.0}.

\begin{lemma}\label{lm:ONE.FOURTH}
Suppose $U$ is a proper, open, connected subset of $\C$, and $f$ is a conformal 
map on $U$.  Then, for all $z \in U$
\[
\frac{1}{4} |f'(z)| d(z, \C - U) \leq d(f(z), f(U)) \leq 4 |f'(z)| d(z, \C - U).
\]
\end{lemma}

Our first goal is to show that any neighborhood that hits the boundary of $D$
must hit it at a finite point (when $D$ is a non-degenerate, finitely connected domain).
We will accomplish this with a little point-set topology.  Recall that 
a point $c$ of
a connected space $K$ is said to be a {\it cut point} of $K$
if $K-\{c\}$ is not connected.

\begin{lemma}\label{lm:BOUNDARY} 
Suppose $Z$ is a connected space with no cut point.
Suppose $O'$ is a subset of $Z$ such that $O'$ and
$Z-O'$ contain at least two points. Then $\partial O'$ contains at least two points.
\end{lemma}

\begin{proof} If $O'\subseteq
Z$ and $\partial O' =\{q\}$, then $Z-\{q\}=[O' - \{q\}]\cup [(Z-O')-\{q\}]$.
On the other hand, $\partial O' = \partial (Z - O')$. 
Hence, $O'$ and $[(Z-O')-\{q\}]$ are open, disjoint, and non-empty.  Thus, $q$ is a
cut point of $Z$- a contradiction.
\end{proof}

\begin{lemma}\label{lm:FINITE} Suppose $D$ is a non-degenerate $n$-connected domain and 
$O \subseteq \XTC$ is an open set such that $O \cap \partial D \neq \emptyset $. Then
there exists a point $p\in \C$ such that $p\in O \cap \partial D$.
\end{lemma}

\begin{proof} Assume that $O\subseteq \XTC$
is open and $\infty \in (\partial D)\cap O$. We can assume 
$O = \XTC - \overline{R}$ for some rational rectangle $R$.  Let
$K_{\infty }$ be the component of $\XTC-D$ that contains
$\infty$.  By connectedness, $K_{\infty }-\{\infty \}$ is not bounded
in $\C$. Therefore, $\C\cap (O\cap K_{\infty
})$ contains at least two points. At the same time, $\C- (O\cap
K_{\infty })$ contains at least two points. The result then follows from 
Lemma \ref{lm:BOUNDARY}.
\end{proof}

\begin{theorem}\label{thm:REVERSAL}
From a name of a finitely connected, non-degenerate domain $D$ that contains $\infty$ but not $0$, the number of its boundary components, a name of $f_D$, and a name of $\partial C_D$, 
one may compute the boundary of $D$.
\end{theorem}

\begin{proof}
By the Extended Computable Open Mapping Theorem, we can compute a name of 
$C_D$ from these data.  It only remains to compute names of 
the boundary of $D$ and $C_D$.  To this end, suppose $R,R', w_0$ are such that 
\begin{eqnarray*}
R, R' & \in & \Q\\
w_0 & \in & D \cap \Q \times \Q\\
R' & \geq & \frac{4R}{|f_D'(w_0)|}\\
R & \geq & d(f_D(w_0), \XTC - C_D)\\
\end{eqnarray*}
The last inequality would be witnessed by the containment in 
$\overline{D_R(f_D(w_0))}$ of a basic neighborhood that hits $C_D$.  Let $z_1 = f_D(w_0)$.
It now follows from Lemma \ref{lm:ONE.FOURTH} that 
\begin{eqnarray}
\nonumber\frac{1}{4} |(f_D^{-1})'(z_1)| d(z_1, \XTC - C_D) & \leq & d(f_D^{-1}(z_1), \XTC - D)\\
& \leq & 4|(f_D^{-1})'(z_1)|d(z_1, \XTC - C_D)\label{eqn:REVERSAL}
\end{eqnarray}
It now follows that $D_{R'}(w_0)$ hits the boundary of $C_D$.  So, once we discover
such $R,R', w_0$, we can begin listing all basic neighborhoods that contain 
the closure of this disk as among those that hit the boundary of $C_D$.  
It follows from Lemma \ref{lm:FINITE} that every basic neighborhood that 
hits the boundary of $C_D$ contains a finite neighborhood that does.  
Since $z_1$ approaches the boundary of $C_D$ as $w_0$ approaches the 
boundary of $D$, it follows from (\ref{eqn:REVERSAL}) that 
there are arbitrarily small disks of the form $D_{R'}(w_0)$.  It now follows
that we can compute a name of the boundary of $C_D$.  It then similarly follows
that we can compute a name of the boundary of $D$.
\end{proof}

It is well-known that there is a computable $1$-connected domain whose boundary is
not computable.

\bibliographystyle{amsplain} 
\bibliography{/Users/tim/myfolders/Lamar/research/bibliographies/computability,/Users/tim/myfolders/Lamar/research/bibliographies/analysis}

\end{document}